\documentclass[a4paper, 11pt]{article}

\usepackage[english]{babel}
\usepackage{amsmath}
\usepackage{amsthm}
\usepackage{amssymb}
\usepackage{graphicx}
\usepackage{hyperref}
\usepackage{enumitem}
\usepackage{enumerate}
\usepackage{subcaption}
\usepackage[utf8]{inputenc}
\usepackage{float}
\usepackage{xcolor}
\usepackage{url}
\usepackage{ragged2e}
\usepackage{authblk}
\newtheorem{defn}{Definition}[section]
\newtheorem{ex}{Example}[section]
\newtheorem{rem}{Remark}[section]
\newtheorem{thm}{Theorem}[section]

\newtheorem{result}{Result}[section]

\title{A Study of the $\ell$-Extension to the Results of Classical Calculus}
\author[1]{Rahul J. Jada}
\author[1*]{Meera H. Chudasama}
\affil[1]{jadarahul19@gmail.com}
\affil[1*]{meera.chudasama@yahoo.co.in}

\affil{Department of Mathematics, Sardar Patel University,    
	\newline Vallabh Vidyanagar, Anand-388120.}
\date{}

\allowdisplaybreaks
\begin{document}
	\maketitle
	\vspace*{-0.7cm}
	\begin{abstract}
		In this paper, new generalizations of the hyper Bessel type differential operator called it as $\ell$-H derivative operator, are studied, and using these, some of the classical results of calculus are extended. Also, new generalization of a natural number is defined and then its combinatorial properties are extended. After this, the $\ell$-H integral is defined, and then its basic properties are assessed and classical results of calculus pertaining to the integral are extended. 
	\end{abstract}
	
	
	\newcommand{\AMSclass}{\textbf{MSC(2020): }}
	\vspace{0.2cm}
	\AMSclass{26A24, 97I50, 26B15, 26A36, 26A46}
	
\section{Introduction}
\hspace{0.6cm}As it is well known, the exponential and trigonometric functions are essential to calculus since they possess eigenfunction property for derivatives and integration \cite{apostol1991calculus}. Calculus also relies heavily on integral and derivative operators \cite{apostol1958mathematical,ThomasFinney}. In this paper, our aim is to generalize some of the results of classical calculus using the generalization of exponential and trigonometric functions because mathematical physics, engineering and other allied areas can not breath without calculus.
	
The following are some terminologies which we will use throughout.
	
Throughout the paper, $\mathbb{N}, \mathbb{N}_0, \mathbb{Z}, \mathbb{R}$ and $\mathbb{C}$ denote, the set of all natural numbers, the set of all non-negative integers, the set of all integers, the set of all real numbers, and the set of all complex numbers respectively.
	
A new generalization of the Gauss hypergeometric function \cite{rainville} is given by Chudasama and Dave namely
$\ell$-H function or $\ell$-Hypergeometric function \cite{chudasama-q-l-hyper,chudasama2017new,lBessel,chudasama2016some,chudasama-q-gen,chudasama2019new} with the aid of confluent hypergeometric function $_1F_1[z]$ \cite{rainville}. The generalized $\ell$-Hypergeometric function \cite{chudasama-q-gen,chudasama2017new,chudasama2019new} defined by Chudasama and Dave given as 
\begin{defn}\label{rHsp}
	For $p$, $r$, $s$ $\in$ $\mathbb{N} \cup \{0\}$, $a_i$, $z$ $\in$ $\mathbb{C}$, $\forall i = 1, 2, \dots, r$ and $b_j, c_k \in \mathbb{C} \setminus \{0, -1, -2, \dots, \}$, $\forall j = 1,2,\dots,s$, $\forall k=1,2,\dots,p$ the generalized $\ell$-Hypergeometric fucntion \cite{chudasama-q-gen,chudasama2017new,chudasama2019new} is denoted and defined by 
	\begin{align*}
		_rH_s^p(\ell;z) &= \ _rH_s^p(\ell;z) \begin{bmatrix}
			a_1, a_2, \dots, a_r; &  & z \\ b_1, b_2, \dots, b_s; & (c_1, c_2, \dots, c_p:\ell);&
		\end{bmatrix}\\
		&= \sum_{n=0}^{\infty}\dfrac{(a_1)_n(a_2)_n\dots(a_r)_n}{(b_1)_n(b_2)_n\dots(b_s)_n \text{ } (c_1)_n^{\ell n}(c_2)_n^{\ell n}\dots(c_p)_n^{\ell n}}\dfrac{z^n}{n!},
	\end{align*}
with $Re(\ell) \geq 0$ and $Re\left( ( c_1+c_2+\dots+c_p)\ell - \dfrac{\ell p}{2}+s-r+1\right)  > 0$.
\end{defn}

In this paper, some classical results of derivatives are extended via $\ell$-H derivative operator \cite{chudasama2016some,chudasama2017new,chudasama2019new}. A generalization of the classical integral is introduced, fundamental results are extended in context of the $\ell$-H exponential function which is a special case of the generalized hypergeometric function $_rH_s^p(\ell;z)$ defined by Chudasama and Dave \cite{chudasama-q-l-hyper,chudasama2017new,lBessel,chudasama2016some,chudasama-q-gen,chudasama2019new} as given in Definition \ref{rHsp}. 
	
One can observe here that since one of the set of denominator parameters going towards infinity together with the summation index along with the parameter \lq$\ell$\rq, certain terminologies viz binomial expansion, derivative operator, integral operator and some functions need to be generalized in this $\ell$-context. We will see these terminologies one by one.  
	
For $Re(\ell) \geq 0$, the $\ell$-H binomial coefficient \cite{lBessel} is defined, for $1 \leq k \leq n$,\\ by 
\begin{equation*}\label{l-bino}
	\binom{n}{k}^{(\ell)} = \dfrac{(n!)^{\ell n+1}}{(k!)^{\ell k+1}((n-k)!)^{\ell n-\ell k+1}}.
\end{equation*}
	
For $n \in \mathbb{N}$, the $\ell$-analogue of the binomial expansion $(z_1 + z_2)^n$ \cite{lBessel}, is given by 
\begin{equation}\label{l-bino th}
	(z_1 + z_2)_{(\ell)}^n = \sum_{k=0}^{n} \binom{n}{k}^{(\ell)} z_1^{n-k} z_2^k,
\end{equation}
with the convention $(z_1 + z_2)_{(\ell)}^0 = 1$. This called the $\ell$-Binomial expansion \cite{lBessel}. Also, it is noted that $(z_1 + z_2)_{(0)}^n = (z_1 + z_2)^n$.

The $\ell$-analogue of the exponential function called it as $\ell$-H exponential function \cite{chudasama2017new} is defined for all $z \in \mathbb{C}$ and $Re(\ell) \geq 0$ as 
\begin{equation}\label{eHl}
	e_H^{\ell}(z) = \sum_{n=0}^{\infty}\dfrac{z^n}{(n!)^{\ell n+1}}.
\end{equation}
Observe that for $\ell=0$
\begin{equation}\label{e^z}
	e_H^{0}(z) = \sum_{n=0}^{\infty}\dfrac{z^n}{n!} = e^z.
\end{equation}
	
\noindent It is noted that \cite{lBessel,chudasama2019new}
\begin{equation}\label{l bino eHl}
	e_H^{\ell} (z_1 + z_2) = e_H^{\ell}(z_1) e_H^{\ell}(z_2).
\end{equation}
With $z_2 = -z_1$, it follows from (\ref{l bino eHl}) that \cite{lBessel,chudasama2019new}
\begin{equation*}\label{eHl0}
	e_H^{\ell}(z_1) e_H^{\ell}(-z_1)= e_H^{\ell}(z_1 - z_1) = e_H^{\ell}(0) = 1,
\end{equation*}
since $e_H^{\ell}(0) = 1$, 
	
Also, Chudasama and Dave have defined following definitions of the $\ell$-H trigonometric functions \cite{chudasama2019new} in the context of the $\ell$-H exponential function.
	\begin{defn}
		For $z \in \mathbb{C}$ and $Re(\ell)\geq0$, the $\ell$-H cosine and $\ell$-H sine functions \cite{chudasama2019new} are denoted and defined by
		\begin{equation}\label{sinlH}
			cos_H^{\ell}(z) = \dfrac{e_H^{\ell}(iz)+e_H^{\ell}(-iz)}{2} =  \sum_{n=0}^{\infty}\dfrac{(-1)^{n}z^{2n}}{((2n)!)^{2\ell n+1}},
		\end{equation}
		and
		\begin{equation}\label{coslH}
			sin_H^{\ell}(z) = \dfrac{e_H^{\ell}(iz)-e_H^{\ell}(-iz)}{2i} =  \sum_{n=0}^{\infty}\dfrac{(-1)^{n}z^{2n+1}}{((2n+1)!)^{2\ell n+\ell+1}}
		\end{equation}
		respectively.
	\end{defn}
\noindent Observe that these two are classical cosine and sine functions when $\ell = 0$.
	
	\begin{rem}
		Now onwards, whenever required, we will follow the following convention:\\
		If $Re(\ell)\geq0$,\text{ } $f_H^{\ell}(z) = \displaystyle \sum_{n=0}^{\infty}a_{n,\ell}\text{ }z^n$, and $z_1, z_2$ are functions of $x$ or functions of $t$ or functions of both $x$ and $t$, then the expansion will take place as
		\begin{equation*}
			f_H^{\ell}(z_1+z_2) = \sum_{n=0}^{\infty}a_{n,\ell}(z_1+z_2)_{(\ell)}^n = \sum_{n=0}^{\infty}a_{n,\ell}\sum_{k=0}^{n}\binom{n}{k}^{(\ell)}z_1^{n-k}z_2^k.
		\end{equation*}
	\end{rem}

	Further, to study properties and behavior of $\ell$-H exponential function and $\ell$-H trigonometric functions, and to generalize results of calculus, we need to recall the differential operator defined by Chudasama and Dave \cite{chudasama2017new}.
	\begin{defn}\label{Delta def}
		Let $f(z) = \displaystyle \sum_{n=1}^{\infty}a_nz^n, 0 \neq z \in \mathbb{C}, \text{ } p\in \mathbb{N}_0$ and $\alpha \in \mathbb{C}$. The operator \cite{chudasama2017new} is denoted and defined by  
		\begin{equation}
			_p\Delta_{\alpha}^{\theta}(f(z)) = \begin{cases}
				\displaystyle  \sum_{n=1}^{\infty}a_n(\alpha)_{n-1}^p(\theta+\alpha-1)^{pn}\left( z^n\right) , & \text{ if } p\in \mathbb{N} \\
				f(z), & \text{ if } p=0
			\end{cases}, 
		\end{equation}
		where $\theta$ is Euler differential operator $\theta \equiv z\dfrac{d}{dz}$ and
		\begin{equation*}
			(\theta + \alpha)^r = \underbrace{(\theta + \alpha)(\theta + \alpha)\dots(\theta + \alpha)}_{r \text{-times}}
		\end{equation*}
		is a special case of the hyper-Bessel differential operators \cite{kiryakova2008transmutation,kiryakova2014hyper}.
	\end{defn}
\noindent Chudasama and Dave, further generalize this operator \cite{chudasama2017new} as,
	
	\begin{defn}\label{lDm oper}
		Let $f(z) = \displaystyle \sum_{n=0}^{\infty}a_nz^n $, $|z|<R$. The operator \cite{chudasama2017new} is denoted and defined by 
		\begin{equation}\label{lDm eq}
			_pD_M^{(z)}\left[ f(z)\right] = z^{-1}\ _p\Delta_1^{\theta}\left( \theta\left( f(z)\right) \right),
		\end{equation}
		where  $z \neq 0, p\in \mathbb{N}_0$ and the operator $_p\Delta_1^{\theta}$ is as defined in Definition \ref{Delta def}.
	\end{defn}
	
	\begin{rem}\label{D ope}
		For $p=0$, we can observe from Definition \ref{Delta def} that,
		\begin{equation*}
			_0\Delta_1^{\theta}\left( f(z)\right) = f(z).
		\end{equation*} 
		Hence,
		\begin{align*}
			_0D_M^{(z)}\left[ f(z)\right] &= z^{-1}\left( \theta\left( f(z)\right) \right)\\
			&= z^{-1} \left( zD\left( f(z)\right)\right) \\
			&= D\left( f(z)\right),
		\end{align*}
		where $D\equiv\dfrac{d}{dz}$.\\
		i.e. For $p=0$, \text{ }$_{p}D_M^{(z)}$ is a classical derivative operator $D$ \cite{ThomasFinney}.
	\end{rem}
	
	\begin{result}
		The operator defined in Definition \ref{lDm oper} is a linear operator \cite{chudasama2017new}. \\
		That is, if $f(z) = \displaystyle \sum_{n=0}^{\infty}a_nz^n$ and $g(z) = \displaystyle \sum_{n=0}^{\infty}b_nz^n$, $|z|<R$, then for $\alpha, \beta \in \mathbb{R}$, 
		\begin{equation*}\label{lDm linear}
			_pD_M^{(z)}\left[ \alpha f(z) + \beta g(z) \right] = \alpha \ _pD_M^{(z)}\left[ f(z)\right] + \beta \ _pD_M^{(z)}\left[ g(z)\right].
		\end{equation*}
	\end{result}
	
	Now, observe that, when the operator defined in Definition \ref{lDm oper} is applied to $\ell$-H exponential function and $\ell$-H trigonometric functions, we get the following respectively \cite{chudasama2017new}. \\
	\begin{result}
		For $\ell \in \mathbb{N}_0$ and fixed $t$ \cite{chudasama2017new},
		\begin{eqnarray*}
			_{\ell}D_M^{(z)}\left[ e_H^{\ell}(zt) \right] = te_H^{\ell}(zt) \hspace{0.2cm}, \\
			\ _{\ell}D_M^{(z)}\left[ cos_H^{\ell}(zt) \right] = -t sin_H^{\ell}(zt) \hspace{0.2cm}, \\
			\ _{\ell}D_M^{(z)}\left[ sin_H^{\ell}(zt) \right] = t cos_H^{\ell}(zt).
		\end{eqnarray*}
	\end{result}
	\begin{proof}
		One can easily verify this by definition. 
	\end{proof}
	
\section{Main Results}

\indent In this section, some terminologies and hence some of the fundamental results of calculus like chain rule, multiplication rule and fundamental theorem of calculus are extended in context to the parameter \lq$\ell$\rq.\\
First of all the $\ell$-analogue of a natural number is defined as follows: 
	\begin{defn}\label{l-num}
		For $n \in \mathbb{N} $ and $Re(\ell) \geq 0$, the $\ell$-analogue of a natural number $n$ is denoted and defined as
		\begin{equation*}
			\left[ n\right]_{\ell} = (n)^{\ell n+1}((n-1)!)^{\ell} = \dfrac{(n!)^{\ell n+1}}{((n-1)!)^{\ell n-\ell+1}},
		\end{equation*}
		with $\left[ 0\right]_{\ell}=0$.
	\end{defn}
	
	\begin{rem}\label{n num}
		For $\ell =0$, formula of the Definition \ref{l-num} reduces to
		\begin{equation*}
			\left[ n\right]_{0} = \dfrac{n!}{(n-1)!} = n.
		\end{equation*}
	\end{rem}
	
	\begin{defn}\label{l-ff}
		For $n \in \mathbb{N}_0$, the $\ell$-H factorial function ($\ell$-H Pochhammer symbol) is denoted and defined as 
		\begin{equation*}
			\left( \left[ n\right]_{\ell} \right)_k = \left[n\right]_{\ell}\left[n+1\right]_{\ell} \left[n+2\right]_{\ell} \dots \left[n+k-2\right]_{\ell} \left[n+k-1\right]_{\ell}.
		\end{equation*}
	\end{defn}
\noindent This leads us to the special case when $n=1$,  
	\begin{equation}\label{l-fac}
		\left( \left[ 1\right]_{\ell} \right)_k = \left[k\right]_{\ell}! := \left[k\right]_{\ell}\left[k-1\right]_{\ell} \left[k-2\right]_{\ell} \dots \left[2\right]_{\ell} \left[1\right]_{\ell},
	\end{equation} 
	with $\left[0\right]_{\ell}! = 1$.
	
\begin{rem}
	For $\ell=0$, by Remark \ref{n num}, Definition \ref{l-ff} reduces to \cite{rainville,saran1989special}
	\begin{align*}
		\left( \left[n\right]_{0}\right)_k &= \left[n\right]_{0}\left[n+1\right]_{0} \left[n+2\right]_{0} \dots \left[n+k-2\right]_{0} \left[n+k-1\right]_{0}\\ &= n(n+1) \dots (n+k-2)(n+k-1) \\ &= \left( n\right)_k,
	\end{align*}
	that is, the Pochammer symbol.
\end{rem}
	
\begin{result}\label{fff}
	For $n \in \mathbb{N}_0$, 
	\begin{equation*}
		\left[n\right]_{\ell}! = (n!)^{\ell n+1}.
	\end{equation*} 
\end{result}
	
\begin{proof}
	From (\ref{l-fac}),
	\begin{equation*}
		\left[n\right]_{\ell}! = \left[n\right]_{\ell}\left[n-1\right]_{\ell} \left[n-2\right]_{\ell} \dots \left[2\right]_{\ell} \left[1\right]_{\ell}.
	\end{equation*}
	Now, using Definition \ref{l-num}, we have
	\begin{align*}
		\left[n\right]_{\ell}! &= \dfrac{(n!)^{\ell n+1}}{((n-1)!)^{\ell n-\ell+1}} \dfrac{((n-1)!)^{\ell n- \ell +1}}{((n-2)!)^{\ell n-2\ell+1}} \dots \dfrac{(2!)^{2\ell +1}}{(1!)^{\ell +1}} \dfrac{(1!)^{\ell +1}}{(0!)^{0\ell +1}}\\
		&= \dfrac{(n!)^{\ell n+1}}{(0!)^{0\ell+1}} = (n!)^{\ell n+1}.
	\end{align*}
\end{proof}
	
\begin{result}\label{fffc}
	For $m,n \in \mathbb{N}_0$ with $m \leq n$,
	\begin{equation*}
		\left( \left[n-m+1\right]_{\ell}\right)_m = \dfrac{(n!)^{\ell n+1}}{((n-m)!)^{\ell n -\ell m+1}}.
	\end{equation*}
\end{result}

\begin{proof}
	From Definition \ref{l-ff},
	\begin{equation*}
		\left( \left[n-m+1\right]_{\ell}\right)_m = \left[n-m+1\right]_{\ell} \left[n-m+2\right]_{\ell} \dots \left[n-1\right]_{\ell} \left[n\right]_{\ell}.
	\end{equation*}
	Multiplying and dividing by $\left[n-m\right]_{\ell}!$, and then using Result \ref{fff}, we get
	\begin{equation*}
		\left( \left[n-m+1\right]_{\ell}\right)_m = \dfrac{ \left[n\right]_{\ell}!}{\left[n-m\right]_{\ell}!}.
	\end{equation*}
	Using Result \ref{fff}, that is
	\begin{equation*}
		\left[n\right]_{\ell}! = (n!)^{\ell n+1} \text{ and } \left[n-m\right]_{\ell}! = ((n-m)!)^{\ell n- \ell m+1},
	\end{equation*}
	we can write
	\begin{equation*}
		\left( \left[n-m+1\right]_{\ell}\right)_m = \dfrac{(n!)^{\ell n+1}}{((n-m)!)^{\ell n -\ell m+1}}.
	\end{equation*}
\end{proof}
	
	The following are the extensions of the classical chain rule and multiplication rule in context of the operator defined in Definition \ref{lDm oper}.
	\begin{defn}\label{l-chain}
		Let $f_{\ell}(x) = \displaystyle \sum_{n=0}^{\infty}a_{n,\ell}\text{ }x^n$ and $k \in \mathbb{N}$. The $\ell$-H chain rule is defined by 
		\begin{equation*}\label{chain eq}
			_{\ell}D_M^{(x)}\left[ f_{\ell}\left( x^k\right) \right] = \left[ k \right]_{\ell}x^{k-1} \ _{\ell}D_M^{(x^k)}\left[ f_{\ell}\left( x^k\right) \right].
		\end{equation*}
	\end{defn}
	
	\begin{rem}
		For $\ell=0$, using Remark \ref{D ope}, i.e., $\left[ k \right]_{0} = k$ and Remark \ref{n num}, i.e., $_{\ell}D_M^{(x)} \equiv D$, Definition \ref{l-chain} reduces to \cite{leibniz2012early}
		\begin{equation*}
			D\left[f\left(x^k\right)\right] = kx^{k-1}D^{\left(x^k\right)}\left[f\left(x^k\right)\right],
		\end{equation*}
		where $D \equiv \dfrac{d}{dx}$ and $D^{\left(x^k\right)} \equiv \dfrac{d}{dx^k}$, which is chain rule for  derivatives.
	\end{rem}
	In the next theorem, for this newly defined $\ell$-H derivative operator, multiplication rule is established.
	\begin{thm}\label{l-multi}
		Let $f_{\ell}(x) = \displaystyle \sum_{n=0}^{\infty}a_{n,\ell}\text{ }x^n$ and $g_{\ell}(x) = \displaystyle \sum_{n=0}^{\infty}b_{n,\ell}\text{ }x^n$. If $\left[ n \right]_{\ell} {+}_{(\ell)} \left[ k \right] _{\ell} = \left[ n {+} k \right] _{\ell}$ when $n,k \in \mathbb{N}_0$, the $\ell$-H multiplication rule is given by
		\begin{equation*}\label{product eq}
			_{\ell}D_M^{(x)}\left[ f_{\ell}(x)  g_{\ell}(x) \right] = g_{\ell}(x)\ _{\ell}D_M^{(x)}\left[ f_{\ell}(x) \right] +_{(\ell)}  f_{\ell}(x)\ _{\ell}D_M^{(x)}\left[ g_{\ell}(x) \right].
		\end{equation*}
	\end{thm}
	
	\begin{proof}
	We will use following series transformations \cite{rainville}:
		\begin{equation}\label{0,0}
			\sum_{n=0}^{\infty}\sum_{k=0}^{\infty} A(k,n) = \sum_{n=0}^{\infty}\sum_{k=0}^{n} A(k,n-k),
		\end{equation}
		\begin{equation}\label{1,0}
			\sum_{n=1}^{\infty}\sum_{k=0}^{\infty} A(k,n) = \sum_{n=1}^{\infty}\sum_{k=0}^{n-1} A(k,n-k)
		\end{equation}
		and
		\begin{equation}\label{0,1}
			\sum_{n=0}^{\infty}\sum_{k=1}^{\infty} A(k,n) = \sum_{n=1}^{\infty}\sum_{k=1}^{n} A(k,n-k).
		\end{equation}
		Consider
		\begin{align*}
			h_{\ell}(x) &= f_{\ell}(x)g_{\ell}(x)\\
			&= \sum_{n=0}^{\infty}a_{n,\ell}\text{ }x^{n} \sum_{k=0}^{\infty}b_{k,\ell}\text{ }x^{k}\\ 
			&= \sum_{n=0}^{\infty} \sum_{k=0}^{\infty}a_{n,\ell}\text{ }b_{k,\ell}\text{ }x^{n+k}.
		\end{align*}
		Using \eqref{0,0}, we have
		\begin{align}
			h_{\ell}(x)&= \sum_{n=0}^{\infty}\left( \sum_{k=0}^{n}a_{n-k,\ell}\text{ }b_{k,\ell}\right) x^{n}\nonumber\\ 
			&= \sum_{n=0}^{\infty}c_{n,\ell}\text{ }x^{n}\label{cnl},
		\end{align}
		where $c_{n,\ell} = \displaystyle \sum_{k=0}^{n}a_{n-k,\ell}\text{ }b_{k,\ell}$.\\
		Then by Definition \ref{lDm oper}, we get
		\begin{align*}
			_{\ell}D_M^{(x)}\left[ h_{\ell}(x) \right] &= \ _{\ell}D_M^{(x)} \left[\sum_{n=0}^{\infty}c_{n,\ell}\text{ }x^{n}\right]\\
			&= \sum_{n=1}^{\infty}\left[ n\right]_{\ell} c_{n,\ell}\text{ }x^{n-1}.
		\end{align*}
		Therefore, from \eqref{cnl}, we can write
		\begin{equation}\label{lD-h(x)}
			_{\ell}D_M^{(x)}\left[ h_{\ell}(x) \right] = \sum_{n=1}^{\infty}\left[ n\right]_{\ell}\left( \sum_{k=0}^{n} a_{n-k,\ell}\text{ }b_{k,\ell}\right) x^{n-1}.
		\end{equation}
		Again from Definition \ref{lDm oper}, we have
		\begin{equation*}
			_{\ell}D_M^{(x)}\left[ f_{\ell}(x) \right] = \sum_{n=1}^{\infty}\left[ n\right]_{\ell}a_{n,\ell}\text{ }x^{n-1}
		\end{equation*}
		and 
		\begin{equation*}
			_{\ell}D_M^{(x)}\left[ g_{\ell}(x) \right] = \sum_{k=1}^{\infty}\left[ k\right]_{\ell}b_{k,\ell}\text{ }x^{k-1}.
		\end{equation*}
		Hence,
		\begin{align*}
			\lefteqn{g_{\ell}(x)\ _{\ell}D_M^{(x)}\left[ f_{\ell}(x) \right] +_{(\ell)} f_{\ell}(x)\ _{\ell}D_M^{(x)}\left[ g_{\ell}(x) \right]}\\ 
			&= \sum_{k=0}^{\infty} \sum_{n=1}^{\infty}\left[ n\right]_{\ell}a_{n,\ell}\text{ }b_{k,\ell}\text{ }x^{n+k-1} +_{(\ell)} \sum_{n=0}^{\infty}\sum_{k=1}^{\infty}\left[ k\right]_{\ell}a_{n,\ell}\text{ }b_{k,\ell}\text{ }x^{n+k-1}.
		\end{align*}
		With the aid of \eqref{1,0}, we get
		\begin{align*}
			\lefteqn{g_{\ell}(x)\ _{\ell}D_M^{(x)}\left[ f_{\ell}(x) \right] +_{(\ell)} f_{\ell}(x)\ _{\ell}D_M^{(x)}\left[ g_{\ell}(x) \right]}\\ 
			&= \sum_{n=1}^{\infty}\sum_{k=0}^{n-1} \left[ n-k\right]_{\ell}a_{n-k,\ell}\text{ }b_{k,\ell}\text{ }x^{n-1} +_{(\ell)} \sum_{n=0}^{\infty}\sum_{k=1}^{\infty}\left[ k\right]_{\ell}a_{n,\ell}\text{ }b_{k,\ell}\text{ }x^{n+k-1}.
		\end{align*}
		Now, using \eqref{0,1}, we get
		\begin{align*}
			\lefteqn{g_{\ell}(x)\ _{\ell}D_M^{(x)}\left[ f_{\ell}(x) \right] +_{(\ell)} f_{\ell}(x)\ _{\ell}D_M^{(x)}\left[ g_{\ell}(x) \right]}\\ 
			&= \sum_{n=1}^{\infty}\sum_{k=0}^{n-1} \left[ n-k\right]_{\ell}a_{n-k,\ell}\text{ }b_{k,\ell}\text{ }x^{n-1} +_{(\ell)} \sum_{n=1}^{\infty}\sum_{k=1}^{n}\left[ k\right]_{\ell}a_{n-k,\ell}\text{ }b_{k,\ell}\text{ }x^{n-1}\\
			&= \sum_{n=1}^{\infty}\left[ \left[ n\right]_{\ell}a_{n,\ell}\text{ }b_{0,\ell}  + \sum_{k=1}^{n-1} \left(\left[ n-k\right]_{\ell} +_{(\ell)} \left[ k\right]_{\ell}\right) a_{n-k,\ell}\text{ }b_{k,\ell} + \left[ n\right]_{\ell}a_{0,\ell}\text{ }b_{n,\ell}\right]x^{n-1}\\
			&= \sum_{n=1}^{\infty}\left[ \left[ n\right]_{\ell}a_{n,\ell}\text{ }b_{0,\ell}  + \sum_{k=1}^{n-1} \left[ n\right]_{\ell} a_{n-k,\ell}\text{ }b_{k,\ell} + \left[ n\right]_{\ell}a_{0,\ell}\text{ }b_{n,\ell}\right]x^{n-1}\\
			&= \sum_{n=1}^{\infty} \left[ n\right]_{\ell} \left( \sum_{k=0}^{n} a_{n-k,\ell}\text{ }b_{k,\ell}\right) x^{n-1},
		\end{align*}
		by hypothesis condition. Hence, from \eqref{lD-h(x)}, we can write
		\begin{equation*}
			g_{\ell}(x)\ _{\ell}D_M^{(x)}\left[ f_{\ell}(x) \right] +_{(\ell)} f_{\ell}(x)\ _{\ell}D_M^{(x)}\left[ g_{\ell}(x) \right] = \ _{\ell}D_M^{(x)}\left[ h_{\ell}(x) \right] = \ _{\ell}D_M^{(x)}\left[ f_{\ell}(x)g_{\ell}(x) \right]
		\end{equation*}	
	\end{proof}

\begin{rem}
	For $\ell=0$, Theorem \ref{l-multi} reduces to \cite{leibniz2012early}
	\begin{equation*}
		D\left[ f(x)g(x)\right] =  D\left[ f(x) \right] g(x) + f(x) D\left[ g(x)\right],
	\end{equation*}
	that is, multiplication rule for derivatives.
\end{rem}

For the justification of the hypothesis of Theorem \ref{l-multi}, we have chosen following example.
	
	\begin{ex}
		Let $f(x) = exp(2x)$ and $g(x) = exp(3x)$.\\
		For the choices of $f(x)$ and $g(x)$ we will prove:
		\begin{equation*}
			_{\ell}D_M^{(x)}\left[ exp(5x) \right] =  exp(3x)\ _{\ell}D_M^{(x)}\left[ exp(2x) \right] +_{(\ell)}  exp(2x)\ _{\ell}D_M^{(x)}\left[ exp(3x) \right]
		\end{equation*}
	\end{ex}
\noindent\textbf{Solution}:
		For $|x|<\infty$,
		\begin{equation*}
			exp(5x) = \sum_{n=0}^{\infty}\dfrac{(5x)^{n}}{n!}.
		\end{equation*}
		Applying an operator defined in Definition \ref{l-num} to $exp(5x)$, we get
		\begin{align}
			_{\ell}D_M^{(x)}\left[ exp(5x) \right] &= \ _{\ell}D_M^{(x)}\left[ \sum_{n=0}^{\infty}\dfrac{(5x)^{n}}{n!} \right]\nonumber\\
			&= \sum_{n=1}^{\infty}\dfrac{\left[ n\right]_{\ell}5^nx^{n-1}}{n!}\label{lD-5x}.
		\end{align}
		Again, when applying $_{\ell}D_M^{(x)}$ as defined in Definition \ref{l-num} to $exp(2x)$ and $exp(3x)$, we have
		\begin{equation*}
			_{\ell}D_M^{(x)}\left[ exp(2x) \right]
			= \sum_{n=1}^{\infty}\dfrac{\left[ n\right]_{\ell}2^nx^{n-1}}{n!}
		\end{equation*}
		and 
\begin{equation*}
	_{\ell}D_M^{(x)}\left[ exp(3x) \right]
	= \sum_{k=1}^{\infty}\dfrac{\left[ k\right]_{\ell}3^kx^{k-1}}{k!},
\end{equation*}
respectively.\\ \\
Hence, from these expressions,
\begin{align*}
	\lefteqn{exp(3x)\ _{\ell}D_M^{(x)}\left[ exp(2x) \right] +_{(\ell)}  exp(2x)\ _{\ell}D_M^{(x)}\left[ exp(3x) \right]} \\
	&= \sum_{k=0}^{\infty}\dfrac{3^kx^{k}}{k!} 
	\sum_{n=1}^{\infty}\dfrac{\left[ n\right]_{\ell}2^nx^{n-1}}{n!} +_{(\ell)} \sum_{n=0}^{\infty}\dfrac{2^nx^{n}}{n!} \sum_{k=1}^{\infty}\dfrac{\left[ k\right]_{\ell}3^kx^{k-1}}{k!}\\
	&= \sum_{n=1}^{\infty} \sum_{k=0}^{\infty}
	\dfrac{\left[ n\right]_{\ell}2^n3^k}{k!n!}x^{n+k-1} +_{(\ell)} \sum_{n=0}^{\infty} \sum_{k=1}^{\infty}\dfrac{\left[ k\right]_{\ell}2^n3^k}{k!n!}x^{n+k-1}\\
	&= \sum_{n=1}^{\infty} \sum_{k=0}^{n-1}
	\dfrac{\left[ n-k\right]_{\ell}2^{n-k}3^k}{k!(n-k)!}x^{n-1} +_{(\ell)} \sum_{n=1}^{\infty} \sum_{k=1}^{n}\dfrac{\left[ k\right]_{\ell}2^{n-k}3^k}{k!(n-k)!}x^{n-1}\\
	&= \sum_{n=1}^{\infty} \left[\dfrac{\left[n\right]_{\ell}2^{n}}{n!} + \sum_{k=1}^{n-1}
	\left(\left[n-k\right]_{\ell} +_{(\ell)} \left[ k\right]_{\ell}\right)\dfrac{2^{n-k}3^k}{k!(n-k)!} + \dfrac{\left[n\right]_{\ell}3^{n}}{n!}\right]x^{n-1}\\
	&= \sum_{n=1}^{\infty} \left[\dfrac{\left[n\right]_{\ell}2^{n}}{n!} + \sum_{k=1}^{n-1}
	\left[n\right]_{\ell} \dfrac{2^{n-k}3^k}{k!(n-k)!} + \dfrac{\left[n\right]_{\ell}3^{n}}{n!}\right]x^{n-1}\\
	&= \sum_{n=1}^{\infty} \left[n\right]_{\ell}\left[\dfrac{2^{n}}{n!} + \sum_{k=1}^{n-1}
	\dfrac{2^{n-k}3^k}{k!(n-k)!} + \dfrac{3^{n}}{n!}\right]x^{n-1}\\
	&= \sum_{n=1}^{\infty}\left[n\right]_{\ell} \left[\sum_{k=0}^{n}\dfrac{1}{k!(n-k)!}2^{n-k}3^k\right] x^{n-1}\\
	&= \sum_{n=1}^{\infty}\dfrac{\left[n\right]_{\ell}}{n!} \left[\sum_{k=0}^{n}\dfrac{n!}{k!(n-k)!}2^{n-k}3^k\right] x^{n-1}.
\end{align*}
Using binomial expansion, that is
\begin{equation*}
	(5)^{n} = (2+3)^{n} = \sum_{k=0}^{\infty}\dfrac{n!}{k!(n-k)!}2^{n-k}3^k,
\end{equation*}
we can write
\begin{align*}
	exp(3x)\ _{\ell}D_M^{(x)}\left[ exp(2x) \right] +_{(\ell)}  exp(2x)\ _{\ell}D_M^{(x)}\left[ exp(3x) \right]
	&= \sum_{n=1}^{\infty}\dfrac{\left[n\right]_{\ell}}{n!} \left[\left(2+3 \right)^n\right] x^{n-1}\\
	&= \sum_{n=1}^{\infty}\dfrac{\left[n\right]_{\ell}5^n}{n!}x^{n-1}.
\end{align*}
Hence, from \eqref{lD-5x}, we have
\begin{align*}
	exp(3x)\ _{\ell}D_M^{(x)}\left[ exp(2x) \right] +_{(\ell)}  exp(2x)\ _{\ell}D_M^{(x)}\left[ exp(3x) \right] 
	&= \ _{\ell}D_M^{(x)}\left[ exp(5x) \right].
\end{align*}\vspace{0.2cm}

The $\ell$-H exponential function \cite{chudasama2019new} is eigen function with respect to the operator defined in Definition \ref{lDm oper}. That is
\begin{thm}\label{Dx of GF}
	If $\ell \in \mathbb{N}_0$, $x,t \in \mathbb{R}$, then the $\ell$-H exponential function $e_H^{\ell}\left( \left[ 2\right]_{\ell}xt-t^2\right) $ with $t$ fixed, is an eigen function with respect to the operator $_{\ell}D_M^{(x)}$ defined in (\ref{lDm eq}). That is,
	\begin{equation*}\label{lDm GF}
		_{\ell}D_M^{(x)}\left( e_H^{\ell}\left( \left[ 2\right]_{\ell}xt-t^2\right) \right) = \left[ 2\right]_{\ell}\ t\ e_H^{\ell}\left( \left[ 2\right]_{\ell}xt-t^2\right) 
	\end{equation*} 
	for fixed $t$.
\end{thm}

\begin{proof}
	Applying the property given in (\ref{l bino eHl}), we can write
	\begin{equation*}
		e_H^{\ell}\left( \left[ 2\right]_{\ell}xt-t^2\right)  = e_H^{\ell}\left( \left[ 2\right]_{\ell}xt\right)  e_H^{\ell}\left( -t^2\right).
	\end{equation*}
	Further, with the aid of (\ref{lDm eq}), we have
	\begin{eqnarray*}
		_{\ell}D_M^{(x)}\left( e_H^{\ell}\left( \left[ 2\right]_{\ell}xt-t^2\right) \right) &=& e_H^{\ell}\left(-t^2\right)\ _{\ell}D_M^{(x)}\left( e_H^{\ell}(\left[ 2\right]_{\ell}xt)\right) \\
		&=& \left[ 2\right]_{\ell}te_H^{\ell}\left( \left[ 2\right]_{\ell}xt-t^2\right). 
	\end{eqnarray*}
	i.e. The $\ell$-H exponential function is an eigen function with respect to the operator $_{\ell}D_M^{(x)}$.
\end{proof}

Next, we subsequently define an $\ell$-analogue of the integral operator \cite{ThomasFinney} and investigate its some of the properties.
	
	\begin{defn}\label{nebla}
		Let $f_{p}\left(x^{\delta}\right) = \displaystyle \sum_{n=0}^{\infty}a_{n,p}\text{ }x^{\delta n}$, $0 \neq x\in \mathbb{R}$, $p, \delta \in \mathbb{N}_0 $, $\alpha \in \mathbb{C}$, with $\alpha \neq 0, -1, -2, \dots,$ and $\alpha \neq -\delta n$ for $n \in \mathbb{N}_0$. Define
		\begin{align}
			_p\nabla_{\alpha}^{\delta}\left(f_{p}\left(x^{\delta}\right)\right) &=\ _p\nabla_{\alpha}^{\delta}\left(\displaystyle \sum_{n=0}^{\infty}a_{n,p}\text{ }x^{\delta n}\right)\nonumber\\ 
			&= \begin{cases}
				\displaystyle  \sum_{n=0}^{\infty}a_{n,p}\dfrac{1}{(\alpha)_{\delta n}^p(\delta n+\alpha)^{p(\delta n+1)}} x^{\delta n}, & \text{if } p\in \mathbb{N} \\
				f(x), & \text{if } p=0 \label{nabla eq}
			\end{cases}.
		\end{align}
	\end{defn}
	
	\begin{defn}\label{lIm}
		Let $f_{p}(x) = \displaystyle \sum_{n=0}^{\infty}a_{n,p}\text{ }x^{\delta n}$, $0 \neq x\in \mathbb{R}$, $p, \delta \in \mathbb{N}_0$, $a,b \in \mathbb{R}$ with $a \leq b$ and $\alpha \in \mathbb{C}$, with $\alpha \neq -\delta n$, $\alpha \neq 0, -1, -2, \dots,$. Then, the $p$-integral over $a$ to $b$ is denoted and defined as
		\begin{align}\label{lIm eq}
			\int_{(p,a)}^{b}f_{p}(x)\,dx &= \left[ x\ _p\nabla_1^{\delta}\left( \Lambda \left( f_{p}(x)\right) \right)\right]_a^b\\ 
			&= \left[ b\ _p\nabla_1^{\delta}\left( \Lambda \left( f_{p}(b)\right) \right)\right] -\left[ a\ _p\nabla_1^{\delta}\left( \Lambda \left( f_{p}(a)\right) \right)\right],\nonumber
		\end{align}
		where $\Lambda \left( f_{p}(x)\right) = x^{-1}\int f_{p}(x)\,dx$ and $_p\nabla_1^{\delta}$ is as defined in Definition \ref{nebla}.
	\end{defn}
	
	\begin{rem}
		\begin{itemize}
			\item[1.] For $p=0$, we can observe from Definition \ref{nebla} that
			\begin{equation*}
				_0\nabla_1^{\delta}\left( f_{p}(x)\right) = f(x).
			\end{equation*} \noindent Hence, $p$-integral reduces to classical integration over $[a,b]$ when $p=0$, since
			\begin{align*}
				\int_{(0,a)}^{b} f(x)\,dx 
				&= \left[ x\ _0\nabla_1^{\delta}\left( \Lambda \left( f(x)\right) \right)\right]_a^b\\
				&= \left[ x\ _0\nabla_1^{\delta}\left(x^{-1} \int f(x)\,dx  \right)\right]_a^b\\
				&= \left[ x \left( x^{-1} \int f(x)\,dx \right)\right]_a^b\\
				&= \int_{a}^{b}f(x)\,dx.
			\end{align*}\label{l-int(0)}
			\item[2.] When $b=a$, we can observe from Definition \ref{lIm} that
			\begin{align*}
				\displaystyle \int_{(p,a)}^{a}f_{p}(x)\,dx &= \left[ a\ _p\nabla_1^{\delta}\left( \Lambda \left( f_{p}(a)\right) \right)\right] -\left[ a\ _p\nabla_1^{\delta}\left( \Lambda \left( f_{p}(a)\right) \right)\right]\\ &= 0.
			\end{align*}
			
			\item[3.] The operator defined in Definition \ref{lIm} is linear. That is, for $f_{p}(x) = \displaystyle \sum_{n=0}^{\infty}a_{n,p}\text{ }x^n$ and $g_{p}(x) = \displaystyle \sum_{n=0}^{\infty}b_{n,p}\text{ }x^n$,
			\begin{equation*}\label{lIm linear}
				\int_{(p,a)}^{b}\left( \alpha f_{p}(x) + \beta g_{p}(x) \right)\,dx = \alpha \left( \int_{(p,a)}^{b}f_{p}(x)\,dx\right)  + \beta \left( \int_{(p,a)}^{b}g_{p}(x)\,dx\right),
			\end{equation*}
			for $\alpha, \beta \in \mathbb{R}$.
		\end{itemize}
	\end{rem}

	
	%
	
	Using this in the next theorem, we have proved the $\ell$-analogue of the fundamental theorem of integral calculus \cite{apostol1991calculus,apostol1958mathematical}. 
	\begin{thm}[\textbf{$\ell$-Analogue of the Fundamental Theorem of Calculus}] \label{l-FTC}
		If $f_{\ell}(x) = \ _{\ell}D_M^{(x)}\left[ F_{\ell}(x)\right]$, where $F_{\ell}(x)$ is of the form $\displaystyle \sum_{n=0}^{\infty} a_{n,\ell}\text{ } x^n$ on $[a,b]$, then 
		\begin{equation*}\label{FTC}
			\int_{(\ell,a)}^{b}f_{\ell}(t)\,dt = F_{\ell}(b) - F_{\ell}(a).
		\end{equation*}
	\end{thm}
	
	\begin{proof}
		Define
		\begin{equation*}
			G_{\ell}(x) = \int_{(\ell,a)}^{x}f_{\ell}(t)\,dt.
		\end{equation*}
		Observe that, when $F_{\ell}(x)$ is of the form $\displaystyle \sum_{n=0}^{\infty} a_{n,\ell}\text{ }x^n$ and $G_{\ell}(x)$ is of the form $\displaystyle \sum_{n=0}^{\infty} b_{n,\ell}\text{ }x^n$, then by Definition \ref{lDm oper}
		\begin{equation*}
			_{\ell}D_M^{(x)}\left( F_{\ell}(x)\right) = \sum_{n=1}^{\infty} \left[n\right]_{\ell}a_{n,\ell}\text{ }x^{n-1}
		\end{equation*}
		and 
		\begin{equation*}
			_{\ell}D_M^{(x)}\left( G_{\ell}(x)\right) = \sum_{n=1}^{\infty} \left[n\right]_{\ell}b_{n,\ell}\text{ }x^{n-1}.
		\end{equation*}
		If $_{\ell}D_M^{(x)}\left( F_{\ell}(x)\right) =\ _{\ell}D_M^{(x)}\left( G_{\ell}(x)\right)$, comparing coefficients of $x^{n-1}$, we have 
		\begin{equation*}
			a_{n,\ell} = b_{n,\ell}, \text{ } n \in \mathbb{N}.
		\end{equation*}
		Hence, we can say that, $F_{\ell}$ and $G_{\ell}$ is only differ by a constant, i.e., $F_{\ell}(x) - G_{\ell}(x) = c$ throughout $[a,b]$ for some constant $c$.
		
		\noindent Hence, 
		\begin{align*}
			F_{\ell}(b) - F_{\ell}(a) &= \left( G_{\ell}(b)+c \right) - \left( G_{\ell}(a)+c \right) \nonumber\\
			&= \int_{(\ell,a)}^{b}f_{\ell}(t)\,dt - \int_{(\ell,a)}^{a}f_{\ell}(t)\,dt \nonumber\\
			&= \int_{(\ell,a)}^{b}f_{\ell}(t)\,dt.
		\end{align*}
	\end{proof}
	
	\begin{rem}
		Using Remark \ref{l-int(0)}, 
		\begin{equation*}
			\ \int_{(0,a)}^{b}f_{\ell}(x)\,dx = \int_{a}^{b}f(x)\,dx,
		\end{equation*}
		and hence for $\ell=0$, Theorem \ref{l-FTC} reduces to \cite{ThomasFinney} the Fundamental Theorem of Calculus for the functions of the form $f_0(x) = f(x) = \displaystyle \sum_{n=0}^{\infty} a_{n}\text{ }x^n$. 
	\end{rem}
	
Next, to obtain the $\ell$-analogue of the Gaussian integral \cite{abramowitz1948handbook} or Gaussian integral in context of the $\ell$-H exponential function\cite{chudasama2016some}, we need to generalize following definitions and remarks corresponding to Jacobian\cite{mathai1997jacobians}.

\begin{defn}\label{l-Jacob}
	The $\ell$-H Jacobian, for the transformation $x = f(r, \theta)$ and $y = g(r, \theta)$ is defined by
	\begin{equation*}
		J_{\ell} = \begin{bmatrix}
			_{\ell}{\partial}_M^{(r)}(x) & _{\ell}{\partial}_M^{({\theta})}(x) \\ \\ _{\ell}{\partial}_M^{(r)}(y) & _{\ell}{\partial}_M^{({\theta})}(y)
		\end{bmatrix},
	\end{equation*}
	where $_{\ell}{\partial}_M^{(z)}$ is an operator same as $_{\ell}D_M^{(z)}$ in Definition \ref{lDm oper} with $\theta \equiv z\dfrac{\partial}{\partial z}$.
\end{defn}

\begin{rem}
	For $\ell = 0$, using Remark \ref{D ope}, we can observe that the operator $_{\ell}{\partial}_M^{(z)}$ reduces to $\dfrac{\partial}{\partial z}$. 
	Therefore, Definition \ref{l-Jacob} reduces to \cite{mathai1997jacobians}
	\begin{equation*}
		J_0 = J = \begin{bmatrix}
			\dfrac{\partial x}{\partial r} &\dfrac{\partial x}{\partial \theta} \vspace{0.1cm}\\ \dfrac{\partial y}{\partial r} & \dfrac{\partial y}{\partial \theta}
		\end{bmatrix},
	\end{equation*}
	that is, the Jacobian.
\end{rem}
\begin{rem}\label{jacob}
	Also, when $f_{\ell}(r, \phi) = r cos_H^{\ell}(\phi)$ and $g_{\ell}(r, \phi) = r sin_H^{\ell}(\phi)$, we can observe that 
	\begin{equation*}
		|J_{\ell}| =  r\left(  cos_H^{\ell}(\phi)\right)^2 + r\left(  sin_H^{\ell}(\phi)\right)^2 = r.
	\end{equation*}
\end{rem}

Now, in the next theorem, we will have the value of the $\ell$-H Gaussian integral.
\begin{thm}[\textbf{$\ell$-H Gaussian Integral}]\label{l-Gaussian Int}
	For $Re(\ell) \geq 0$,
	\begin{equation*}\label{GI}
		\int_{(\ell,-\infty)}^{\infty}e_H^\ell\left(-x^2\right)\,dx = \dfrac{\sqrt{\pi}}{2^\ell}. 
	\end{equation*}
\end{thm}
	
\begin{proof}
	Taking the transformation $x=r cos_H^{\ell}(\theta)$ and $y=r  sin_H^{\ell}(\theta)$, by Remark \ref{jacob}, we have $|J_{\ell}| = r$.
		
	\noindent Now, from \eqref{l bino eHl}, we can write
	\begin{align*}
		\lefteqn{\int_{(\ell,-\infty)}^{\infty}e_H^{\ell}\left(-x^2\right) \,dx  \int_{(\ell,-\infty)}^{\infty}e_H^{\ell}\left(-y^2\right) \,dy} \\
		&= \int_{(\ell,-\infty)}^{\infty}\ \int_{(\ell,-\infty)}^{\infty}e_H^{\ell}\left(-\left(x^2+y^2\right)\right)\,dx \,dy \\
		&= \int_{(\ell,0)}^{\infty}\int_{(\ell,0)}^{2\pi}e_H^\ell\left(-\left(r^2\right)\right)r\,dr\,d\theta\\
		&= \int_{(\ell,0)}^{\infty}e_H^\ell\left(-\left(r^2\right)\right)r\,dr \int_{(\ell,0)}^{2\pi}\,d\theta \\
		&= \dfrac{2\pi}{\left[2\right]_{\ell}} \int_{(\ell,0)}^{\infty}e_H^\ell(-(t))\,dt \nonumber\\
		&= \dfrac{2\pi}{\left[2\right]_{\ell}} \left[ -e_H^\ell(-t)\right]_{0}^{\infty}.
		\end{align*}
	Using Definition \ref{lIm}, we have
	\begin{align}
		\int_{(\ell,-\infty)}^{\infty}e_H^{\ell}\left(-x^2\right) \,dx  \int_{(\ell,-\infty)}^{\infty}e_H^{\ell}\left(-y^2\right) \,dy \nonumber
		&= \dfrac{2\pi}{\left[2\right]_{\ell}} \left[ 	-e_H^\ell(-\infty)+e_H^\ell(0)\right] \nonumber\\
		&= \dfrac{2\pi}{\left[2\right]_{\ell}}\left[-0+1\right]\nonumber\\
		&= \dfrac{2\pi}{\left[2\right]_{\ell}}\label{gi}.
	\end{align}
	Now, with the aid of Definition \ref{l-num}, observe that
	\begin{equation*}
		\dfrac{2\pi}{\left[2\right]_{\ell}} = \dfrac{2\pi}{(2)^{2\ell+1}(1!)^{\ell}} = \dfrac{\pi}{2^{2\ell}}.
	\end{equation*}
	Taking square root on both the sides of \eqref{gi}, to get
	\begin{align*}
		\int_{(\ell,-\infty)}^{\infty}e_H^{\ell}\left(-x^2\right) \,dx &= 	\left(\dfrac{{\pi}}{{2^{2\ell}}}\right)^{\dfrac{1}{2}}\\ 
		&= \dfrac{\sqrt{\pi}}{{2^{\ell}}}.
	\end{align*}
\end{proof}
	
\begin{rem}
	For $\ell = 0$, observe that $\ell$-H Gaussian integral becomes
	\begin{equation*}
		\int_{(0,-\infty)}^{\infty}e_H^{0}\left(-x^2\right) \,dx = \dfrac{\sqrt{\pi}}{2^0}.
	\end{equation*}
	Therefore, from Remark \ref{l-int(0)} and \eqref{e^z}, above expression becomes \cite{abramowitz1948handbook}
	\begin{equation*}
		\int_{-\infty}^{\infty}e^{-x^2} \,dx = \sqrt{\pi},
	\end{equation*}
	which is Gaussian Integral.
\end{rem}

\section{Conclusion}
\indent In this paper, we have reviewed and analyzed hyper Bessel type generalized derivative operator call it as $\ell$-H derivative operator. After this, a new generalization of the natural number ($\ell$-number), is defined, various combinatorial properties are examined. Then we have generalized the classical results of calculus viz chain rule, multiplication rule for the $\ell$-H derivative operator. Further, an $\ell$-analogue of the integral operator i.e. $\ell$-H integral is defined. Also, fundamental theorem of integral calculus is extended via $\ell$-H integral operator. At last, an $\ell$-analogue of Jacobian is defined to extend the Gaussian integral in context of $\ell$-H exponential function via the transformation in terms of $\ell$-H sine and cosine function.
Because orthogonal polynomials have exponential functions, trigonometric functions as their generating functions, and as calculus results are crucial to understand the nature and behavior of orthogonal polynomials, this work offers a generalization in the field of orthogonal polynomials, which makes this research valuable.


\section*{Declarations}

\subsection*{Availability of Data and Materials}
Data sharing is not applicable to this article as no datasets were generated or analyzed during the current study.

\subsection*{Competing Interests}
The authors declare that they have no competing interests.

\subsection*{Funding}
This research received no specific grant from any funding agency in the public, commercial, or not-for-profit sectors.

\subsection*{Authors' Contributions}
All authors contributed to the conception and design of the study, manuscript preparation, and revision. All authors read and approved the final manuscript.

\subsection*{Acknowledgment}
The first author is thankful to the CSIR-UGC for providing financial assistance through a Junior Research Fellowship.

\bibliographystyle{plain}
\bibliography{l-Cal-ref}
	
\end{document}